\documentclass[a4paper,11pt]{amsart}
\usepackage{amssymb}
\usepackage[arrow,curve,matrix]{xy}

\title[Derived equivalences of posets of cluster tilting objects]
{Universal derived equivalences of posets of cluster tilting objects}
\author{Sefi Ladkani}
\address{Einstein Institute of Mathematics, The Hebrew University of Jerusalem, Jerusalem 91904, Israel}
\email{sefil@math.huji.ac.il}

\DeclareMathOperator{\fac}{fac}
\DeclareMathOperator{\add}{add}
\DeclareMathOperator{\ind}{ind}
\DeclareMathOperator{\rep}{rep}
\DeclareMathOperator{\Hom}{Hom}
\DeclareMathOperator{\img}{Im}
\DeclareMathOperator{\Ext}{Ext}
\DeclareMathOperator{\coker}{coker}
\DeclareMathOperator{\supp}{supp}

\newcommand{\wh}{\widehat}

\newcommand{\Fp}{F^{+}}
\newcommand{\Fm}{F^{-}}
\newcommand{\ju}{j^{-1}}
\newcommand{\jls}{j_{!}}
\newcommand{\Qx}{Q \setminus \{x\}}
\newcommand{\vphi}{\varphi}

\newcommand{\cA}{\mathcal{A}}
\newcommand{\cC}{\mathcal{C}}
\newcommand{\cD}{\mathcal{D}}
\newcommand{\cT}{\mathcal{T}}
\newcommand{\cX}{\mathcal{X}}

\theoremstyle{plain}
\newtheorem{theorem}{Theorem}
\newtheorem{prop}[theorem]{Proposition}
\newtheorem{lemma}[theorem]{Lemma}
\newtheorem{cor}[theorem]{Corollary}

\numberwithin{theorem}{section} \numberwithin{equation}{section}

\begin{document}

\begin{abstract}
We show that for two quivers without oriented cycles related by a BGP
reflection, the posets of their cluster tilting objects are related by
a simple combinatorial construction, which we call a flip-flop.

We deduce that the posets of cluster tilting objects of derived
equivalent path algebras of quivers without oriented cycles are
universally derived equivalent. In particular, all Cambrian lattices
corresponding to the various orientations of the same Dynkin diagram
are universally derived equivalent.
\end{abstract}

\maketitle

\section{Introduction}

In this note we investigate the combinatorial relations between the
posets of cluster tilting objects of derived equivalent path algebras,
continuing our work~\cite{Ladkani07t} on the posets of tilting modules
of such algebras.

Throughout this note, we fix an algebraically closed field $k$. Let $Q$
be a finite quiver without oriented cycles and let $\rep Q$ denote the
category of finite dimensional representations of $Q$ over $k$. The
associated cluster category $\cC_Q$ was introduced in~\cite{BMRRT06}
(and in~\cite{CCS06} for the $A_n$ case) as a representation theoretic
approach to the cluster algebras introduced and studied by Fomin and
Zelevinsky \cite{FominZelevinsky02}. It is defined as the orbit
category~\cite{Keller05} of the bounded derived category $\cD^b(Q)$ of
$\rep Q$ by the functor $S \cdot [-2]$ where $S : \cD^b(Q) \to
\cD^b(Q)$ is the Serre functor and $[1]$ is the suspension. The
indecomposables of $\cC_Q$ can be represented by the indecomposables of
$\cD^b(Q)$ in the fundamental domain of $S \cdot [-2]$, hence $\ind
\cC_Q = \ind \rep Q \cup \left\{ P_y[1] \,:\, y \in Q \right\}$ where
$P_y$ are the indecomposable projectives in $\rep Q$.

Cluster tilting theory was investigated in~\cite{BMRRT06}. A basic
object $T \in \cC_Q$ is a \emph{cluster tilting object} if
$\Ext^1_{\cC_Q}(T,T) = 0$ and $T$ is maximal with respect to this
property, or equivalently, the number of indecomposable summands of $T$
equals the number of vertices of $Q$. If $T = M \oplus U$ is cluster
tilting and $M$ is indecomposable, then there exist a unique
indecomposable $M' \neq M$ such that $T' = M' \oplus U$ is cluster
tilting. $T'$ is called the \emph{mutation} of $T$ with respect to $M$.

Denote by $\cT_{\cC_Q}$ the set of all cluster tilting objects.
In~\cite{IngallsThomas06}, a partial order on $\cT_{\cC_Q}$, extending
the partial order on tilting modules introduced
in~\cite{RiedtmannSchofield91}, is defined by $T \leq T'$ if $\fac T
\supseteq \fac T'$. Here, for $M \in \rep Q$, $\fac M$ denotes the full
subcategory of $\rep Q$ consisting of all quotients of finite sums of
copies of $M$, and for $T \in \cT_{\cC_Q}$, $\fac T = \fac \wh{T}$
where $\wh{T} \in \rep Q$ is the sum of all indecomposable summands of
$T$ which are not shifted projectives.

As shown in~\cite{IngallsThomas06}, the map $T \mapsto \fac T$ induces
an order preserving bijection between $(\cT_{\cC_Q}, \leq)$ and the set
of finitely generated torsion classes in $\rep Q$ ordered by reverse
inclusion. Moreover, it is also shown that when $Q$ is Dynkin,
$(\cT_{\cC_Q}, \leq)$ is isomorphic to the corresponding Cambrian
lattice defined in~\cite{Reading06} as a certain lattice quotient of
the weak order on the Coxeter group associated with $Q$.

\medskip

For two partially ordered sets $(X, \leq_X)$, $(Y, \leq_Y)$ and an
order preserving function $f : X \to Y$, define two partial orders
$\leq^f_+$ and $\leq^f_-$ on the disjoint union $X \sqcup Y$ by keeping
the original partial orders inside $X$ and $Y$ and setting
\begin{align*}
x \leq^f_{+} y \Longleftrightarrow f(x) \leq_Y y && y \leq^f_{-} x
\Longleftrightarrow y \leq_Y f(x)
\end{align*}
with no other additional order relations. We say that two posets $Z$
and $Z'$ are related via a \emph{flip-flop} if there exist $X$, $Y$ and
$f: X \to Y$ as above such that $Z \simeq (X \sqcup Y, \leq^f_{+})$ and
$Z' \simeq (X \sqcup Y, \leq^f_{-})$.

Let $x$ be a sink of $Q$ and let $Q'$ be the quiver obtained from $Q$
by a BGP reflection~\cite{BGP73} at $x$, that is, by reverting all the
arrows ending at $x$. Our main result is the following.

\begin{theorem} \label{t:clff}
The posets $\cT_{\cC_Q}$ and $\cT_{\cC_{Q'}}$ are related via a
flip-flop.
\end{theorem}

We give a brief outline of the proof. Let $\cT^x_{\cC_Q}$ denote the
subset of cluster tilting objects containing the simple projective
$S_x$ at $x$ as direct summand. Given $T \in \cT^x_{\cC_Q}$, let $f(T)$
be the mutation of $T$ with respect to $S_x$. In
Section~\ref{sec:clsink} we prove that the function $f : \cT^x_{\cC_Q}
\to \cT_{\cC_Q} \setminus \cT^x_{\cC_Q}$ is order preserving and
moreover
\begin{equation} \label{e:ffp}
\cT_{\cC_Q} \simeq \bigl( \cT^x_{\cC_Q} \sqcup (\cT_{\cC_Q} \setminus
\cT^x_{\cC_Q}), \leq^f_{+} \bigr)
\end{equation}

Similarly, let $\cT^{x[1]}_{\cC_{Q'}}$ be the subset of cluster tilting
objects containing the shifted projective $P'_x[1]$ at $x$ as direct
summand. Given $T \in \cT^{x[1]}_{\cC_{Q'}}$, let $g(T)$ be the
mutation of $T$ with respect to $P'_x[1]$. In
Section~\ref{sec:clsource} we prove that the function $g :
\cT^{x[1]}_{\cC_{Q'}} \to \cT_{\cC_{Q'}} \setminus
\cT^{x[1]}_{\cC_{Q'}}$ is order preserving and moreover
\begin{equation} \label{e:ffm}
\cT_{\cC_{Q'}} \simeq \bigl( \cT^{x[1]}_{\cC_{Q'}} \sqcup
(\cT_{\cC_{Q'}} \setminus \cT^{x[1]}_{\cC_{Q'}}), \leq^g_{-} \bigr)
\end{equation}

In Section~\ref{sec:clBGP} we relate the two isomorphisms given
in~\eqref{e:ffp} and~\eqref{e:ffm} by considering,
following~\cite{Zhu07}, the action of the BGP reflection functor on the
cluster tilting objects. We prove the existence of the following
commutative diagram with horizontal isomorphisms of posets
\[
\xymatrix{ \cT^x_{\cC_Q} \ar[r]^{\simeq} \ar[d]^{f} &
\cT^{x[1]}_{\cC_{Q'}} \ar[d]^{g} \\
\cT_{\cC_Q} \setminus \cT^x_{\cC_Q} \ar[r]^{\simeq} & \cT_{\cC_{Q'}}
\setminus \cT^{x[1]}_{\cC_{Q'}} }
\]
from which we deduce Theorem~\ref{t:clff}. An example demonstrating the
theorem and its proof is given in Section~\ref{sec:clexample}.

\medskip

In our previous work~\cite{Ladkani07t}, we have shown a result
analogous to Theorem~\ref{t:clff} for the posets $\cT_Q$ and $\cT_{Q'}$
of tilting modules, following a similar strategy of proof. However,
there are some important differences.

First, the situation in the cluster tilting case is asymmetric, as the
partition~\eqref{e:ffp} for a sink involves the subset of cluster
tilting objects containing the corresponding simple, while the
corresponding partition of~\eqref{e:ffm} at a source involves the
subset of cluster tilting objects containing the shifted projective. In
contrast, both partitions for the tilting case involve the subset of
tilting modules containing the simple, either at a source or sink. This
asymmetry is inherent in the proof of~\eqref{e:ffm}, which is not the
dual of that of~\eqref{e:ffp}, and also in the analysis of the effect
of the BGP reflection functor.

Second, in the cluster tilting case, the order preserving maps
occurring in the flip-flop construction are from the set containing the
simple (or shifted projective) to its complement, while in the tilting
case, they are in the opposite direction, into the set containing the
simple. As a consequence, a partition with respect to a sink in the
cluster tilting case yields an order of the form $\leq_{+}$, while for
the tilting case it gives $\leq_{-}$.

\medskip

While two posets $Z$ and $Z'$ related via a flip-flop are in general
not isomorphic, they are \emph{universally derived equivalent} in the
following sense; for any abelian category $\cA$, the derived categories
of the categories of functors $Z \to \cA$ and $Z' \to \cA$ are
equivalent as triangulated categories, see~\cite{Ladkani07u}.

It is known~\cite[(I.5.7)]{Happel88} that the path algebras of two
quivers $Q$, $Q'$ without oriented cycles are derived equivalent if and
only if $Q'$ can be obtained from $Q$ by a sequence of BGP reflections
(at sources or sinks). We therefore deduce the following theorem.

\begin{theorem} \label{t:clunider}
Let $Q$ and $Q'$ be two quivers without oriented cycles whose path
algebras are derived equivalent. Then the posets $\cT_{\cC_Q}$ and
$\cT_{\cC_{Q'}}$ are universally derived equivalent.
\end{theorem}

Since for a Dynkin quiver $Q$, the poset $\cT_{\cC_Q}$ is isomorphic to
the corresponding Cambrian lattice, the above theorem can be restated
as follows.

\begin{cor}
All Cambrian lattices corresponding to the various orientations of the
same Dynkin diagram are universally derived equivalent.
\end{cor}

In particular, the incidence algebras of the Cambrian lattices
corresponding to the various orientations the same Dynkin diagram are
derived equivalent, as the universal derived equivalence of two finite
posets implies the derived equivalence of their incidence algebras.

\subsection*{Acknowledgement}

I would like to thank Fr\'{e}d\'{e}ric Chapoton for suggesting a
conjectural version of Theorem~\ref{t:clunider} in the case of Dynkin
quivers and for many helpful discussions.

\section{Cluster tilting objects containing $P_x$}
\label{sec:clsink}

Let $x \in Q$ be a vertex, and denote by $\cT^x_{\cC_Q}$ the subset of
cluster tilting objects containing $P_x$ as direct summand.

\begin{lemma}
Let $M \in \rep Q$. Then $P_x \in \fac M$ if and only if $M$ contains
$P_x$ as a direct summand.
\end{lemma}
\begin{proof}
Assume that $P_x \in \fac M$, and let $q : M^n \twoheadrightarrow P_x$
be a surjection, for some $n \geq 1$. Since $P_x$ is projective, there
exists $j : P_x \to M^n$ such that $qj = 1_{P_x}$. Let $N = \img j =
\img jq$. As $(jq)^2 = jq$, we deduce that $N$ is a direct summand of
$M^n$ and that $j : P_x \to N$ is an isomorphism. Since $P_x$ is
indecomposable, it is also a summand of $M$.
\end{proof}

\begin{cor} \label{c:cTxPx}
Let $T \in \cT_{\cC_Q}$. Then $T \in \cT^x_{\cC_Q}$ if and only if $P_x
\in \fac T$.
\end{cor}

\begin{cor} \label{c:cTxclose}
If $T \in \cT^x_{\cC_Q}$ and $T' \leq T$, then $T' \in \cT^x_{\cC_Q}$,
\end{cor}
\begin{proof}
Let $T \in \cT^x_{\cC_Q}$ and $T' \in \cT_{\cC_Q}$. If $T' \leq T$,
then $P_x \in \fac T \subseteq \fac T'$, hence $T' \in \cT^x_{\cC_Q}$.
\end{proof}

Define a map $f : \cT^x_{\cC_Q} \to \cT_{\cC_Q} \setminus
\cT^x_{\cC_Q}$ as follows. Given $T \in \cT^x_{\cC_Q}$, write $T = P_x
\oplus U$ and set $f(T) = M \oplus U$ where $M$ is the unique other
indecomposable complement of $U$ such that $M \oplus U$ is a cluster
tilting object.

Recall that for a tilting module $T \in \rep Q$, we have $\fac T =
T^{\perp}$ where
\[
T^{\perp} = \left\{ M \in \rep Q \,:\, \Ext^1_Q(T, M) = 0 \right\}
\]

\begin{lemma} \label{l:fTgeT}
Let $T \in \cT^x_{\cC_Q}$. Then $f(T) > T$.
\end{lemma}
\begin{proof}
One could deduce the claim from Lemma~2.32 of~\cite{IngallsThomas06}.
Instead, we shall give a direct proof. Write $T = P_x \oplus U$ and
$f(T) = M \oplus U$. If $M$ is a shifted projective, the claim is
clear. Otherwise, by deleting the vertices of $Q$ corresponding to the
shifted projective summands of $U$, we may and will assume that $P_x
\oplus U$ and $M \oplus U$ are tilting modules. Therefore
\[
\fac (P_x \oplus U) = (P_x \oplus U)^{\perp} = U^{\perp}
\]
where the last equality follows since $P_x$ is projective. As $M \in
U^{\perp}$, we get that $M \in \fac (P_x \oplus U)$, hence $\fac (M
\oplus U) \subseteq \fac (P_x \oplus U)$.
\end{proof}

For the rest of this section, \emph{we assume that the vertex $x$ is a
sink in $Q$}. In this case, $P_x = S_x$ and $\ind \fac S_x = \{S_x\}$.
Moreover, $S_x \not \in \fac M$ for any other indecomposable $M \neq
S_x$, since $\Hom_Q(M, S_x) = 0$.

\begin{lemma} \label{l:facfT}
Let $T \in \cT^x_{\cC_Q}$. Then $\ind \fac f(T) = \ind \fac T \setminus
\{S_x\}$.
\end{lemma}
\begin{proof}
Write $T = S_x \oplus U$ and $f(T) = M \oplus U$. By the preceding
remarks,
\[
\ind \fac T = \ind \fac (S_x \oplus U) = \ind \fac S_x \cup \ind \fac U
\]
is a disjoint union $\{S_x\} \amalg \ind \fac U$. By
Lemma~\ref{l:fTgeT},
\[
\ind \fac f(T) = \ind \fac (M \oplus U) \subseteq \ind \fac (S_x \oplus
U) = \{S_x\} \amalg \ind \fac U ,
\]
therefore $\ind \fac(M \oplus U) = \ind \fac U$, as $S_x \not \in \fac
M$.
\end{proof}

\begin{cor} \label{c:fTgeuniq}
Let $T \in \cT^x_{\cC_Q}$ and $T' \in \cT_{\cC_Q} \setminus
\cT^x_{\cC_Q}$ be such that $T' > T$. Then $T' \geq f(T)$.
\end{cor}
\begin{proof}
By assumption, $\fac T' \subseteq \fac T$. Moreover, $S_x \not \in \fac
T'$, since $T' \not \in \cT^x_{\cC_Q}$. Hence by Lemma~\ref{l:facfT},
$\ind \fac T' \subseteq \ind \fac f(T)$, thus $T' \geq f(T)$.
\end{proof}

\begin{cor} \label{c:ffclsink}
The map $f : \cT^x_{\cC_Q} \to \cT_{\cC_Q} \setminus \cT^x_{\cC_Q}$ is
order preserving and
\[
\cT_{\cC_Q} \simeq \bigl( \cT^x_{\cC_Q} \sqcup (\cT_{\cC_Q} \setminus
\cT^x_{\cC_Q}), \leq^f_{+} \bigr)
\]
\end{cor}
\begin{proof}
If $T, T' \in \cT^x_{\cC_Q}$ are such that $T' \geq T$, then by
Lemma~\ref{l:fTgeT}, $f(T') > T' \geq T$, hence by
Corollary~\ref{c:fTgeuniq}, $f(T') \geq f(T)$, therefore $f$ is order
preserving. The other assertion follows from
Corollaries~\ref{c:cTxclose},~\ref{c:fTgeuniq} and Lemma~\ref{l:fTgeT}.
\end{proof}

\section{Cluster tilting objects containing $P_x[1]$}
\label{sec:clsource}

For $M \in \rep Q$ and $y \in Q$, let $M(y)$ denote the vector space
corresponding to $y$, and let $\supp M = \{y \in Q \,:\, M(y) \neq 0
\}$ be the \emph{support} of $M$.

Let $x \in Q$ be a vertex, and denote by $\cT^{x[1]}_{\cC_Q}$ the
subset of cluster tilting objects containing the shifted indecomposable
projective $P_x[1]$ as direct summand.

\begin{lemma} \label{l:cTxopen}
If $T \in \cT^{x[1]}_{\cC_Q}$ and $T' \geq T$, then $T' \in
\cT^{x[1]}_{\cC_Q}$.
\end{lemma}
\begin{proof}
Since $T$ contains $P_x[1]$ as summand, we have $\Ext^1_{\cC_Q}(P_x[1],
T) = 0$, that is, $\Hom_Q(P_x, \wh{T}) = 0$, or equivalently $x \not
\in \supp \wh{T}$.

Now let $T' \geq T$. Then all the modules in $\fac T' \subseteq \fac T$
are not supported on $x$, and in particular $\Hom_Q(P_x, \wh{T'}) = 0$,
thus $\Ext^1_{\cC_Q}(P_x[1], T') = 0$. The maximality of $T'$ implies
that it contains $P_x[1]$ as summand.
\end{proof}

Similarly to the previous section, define a map $g : \cT^{x[1]}_{\cC_Q}
\to \cT_{\cC_Q} \setminus \cT^{x[1]}_{\cC_Q}$ as follows. Given $T \in
\cT^{x[1]}_{\cC_Q}$, write $T = P_x[1] \oplus U$ and set $g(T) = M
\oplus U$ where $M$ is the unique other indecomposable complement of
$U$ such that $M \oplus U$ is a cluster tilting object.

\begin{lemma} \label{l:gTltT}
Let $T \in \cT^{x[1]}_{\cC_Q}$. Then $g(T) < T$.
\end{lemma}
\begin{proof}
This is obvious. Indeed, write $T = P_x[1] \oplus U$ and $g(T) = M
\oplus U$. Then $\fac g(T) = \fac (M \oplus U) \supseteq \fac U = \fac
T$.
\end{proof}

For the rest of this section, \emph{we assume that the vertex $x$ is a
source}. In this case, for any module $M \in \rep Q$, we have that $S_x
\in \fac M$ if and only if $M$ is supported at $x$. Therefore we deduce
the following lemma, which can be viewed as an analogue of
Corollary~\ref{c:cTxPx}.

\begin{lemma}
Let $T \in \cT_{\cC_Q}$. Then $T \in \cT^{x[1]}_{\cC_Q}$ if and only if
$S_x \not \in \fac T$.
\end{lemma}

Recall that a basic module $U \in \rep Q$ is an \emph{almost complete
tilting module} if $\Ext^1_Q(U,U)=0$ and the number of indecomposable
summands of $U$ equals the number of vertices of $Q$ less one. A
\emph{complement} to $U$ is an indecomposable $M$ such that $M \oplus
U$ is a tilting module. It is known~\cite{HappelUnger89} that an almost
complete tilting module $U$ has at most two complements, and exactly
two if and only if $U$ is sincere, that is, $\supp U = Q$.

\begin{prop} \label{p:UMtorsion}
Let $U$ be an almost complete tilting module of $\rep Q$ not supported
on $x$, and let $M$ be its unique indecomposable complement to a
tilting module. Let $\cX$ be a torsion class in $\rep Q$ satisfying
$\fac U \subseteq \cX$ and $S_x \in \cX$. Then $M \in \cX$.
\end{prop}
\begin{proof}
The natural inclusion $j : \Qx \to Q$ induces a pair $(\jls, \ju)$ of
exact functors
\begin{align*}
\ju : \rep Q \to \rep (\Qx) && \jls : \rep (\Qx) \to \rep Q
\end{align*}
where $\ju$ is the natural restriction and $\jls$ is its left adjoint,
defined as the extension of a representation of $\Qx$ by zero at $x$.

Now $\jls \ju U \simeq U$ since $U$ is not supported on $x$. By
adjunction and exactness,
\[
\Ext^1_{\Qx}(\ju U, \ju U) \simeq \Ext^1_Q(\jls \ju U, U) = \Ext^1_Q(U,
U)
\]
thus $\ju U$ is a (basic) tilting module of $\rep (\Qx)$. However,
by~\cite[Proposition~2.6]{Ladkani07t}, $\ju(M \oplus U)$ is also a
tilting module of $\rep (\Qx)$, but not necessarily basic. It follows
that $\ju M \in \add \ju U$, hence $\jls \ju M \in \add \jls \ju U =
\add U$.

The adjunction morphism $\jls \ju M \to M$ is injective, and we have an
exact sequence
\[
0 \to \jls \ju M \to M \to S_x^n \to 0
\]
for some $n \geq 0$. Now $S_x \in \cX$ by assumption and $\jls \ju M
\in \add U \subseteq \cX$, hence $M \in \cX$ as $\cX$ is closed under
extensions.
\end{proof}

\begin{cor} \label{c:gTleuniq}
Let $T \in \cT^{x[1]}_{\cC_Q}$ and $T' \in \cT_{\cC_Q} \setminus
\cT^{x[1]}_{\cC_Q}$ be such that $T' < T$. Then $T' \leq g(T)$.
\end{cor}
\begin{proof}
Write $T = P_x[1] \oplus U$ and $g(T) = M \oplus U$. The assumptions on
$T'$ imply that $S_x \in \fac T'$ and $\fac U = \fac T \subseteq \fac
T'$.

By deleting the vertices of $Q$ corresponding to the shifted projective
summands of $U$, we may and will assume that $M \oplus U$ is a tilting
module, so that $U$ is an almost complete tilting module. Applying
Proposition~\ref{p:UMtorsion} for $\cX = \fac T'$, we deduce that $M
\in \fac T'$, hence $\fac g(T) = \fac(M \oplus U) \subseteq \fac T'$.
\end{proof}

\begin{cor} \label{c:ffclsource}
The map $g : \cT^{x[1]}_{\cC_Q} \to \cT_{\cC_Q} \setminus
\cT^{x[1]}_{\cC_Q}$ is order preserving and
\[
\cT_{\cC_Q} \simeq \bigl( \cT^{x[1]}_{\cC_Q} \sqcup (\cT_{\cC_Q}
\setminus \cT^{x[1]}_{\cC_Q}), \leq^g_{-} \bigr)
\]
\end{cor}
\begin{proof}
The claim follows from Lemmas~\ref{l:cTxopen}, \ref{l:gTltT} and
Corollary~\ref{c:gTleuniq} as in the proof of
Corollary~\ref{c:ffclsink}.
\end{proof}

\section{The effect of a BGP reflection}
\label{sec:clBGP}

Let $Q$ be a quiver without oriented cycles and let $x$ be a sink. Let
$y_1, \dots, y_m$ be the endpoints of the arrows ending at $x$, and
denote by $Q'$ the quiver obtained from $Q$ by reflection at $x$. For a
vertex $y \in Q$, denote by $S_y$, $S'_y$ the simple modules
corresponding to $y$ in $\rep Q$, $\rep Q'$ and by $P_y$, $P'_y$ their
projective covers.

The categories $\rep Q$ and $\rep Q'$ are related by the BGP reflection
functors, introduced in~\cite{BGP73}. We recollect here the basic facts
on these functors that will be needed in the sequel.

The BGP reflection functors are the functors
\begin{align*}
\Fp : \rep Q \to \rep Q' && \Fm : \rep Q' \to \rep Q
\end{align*}
defined by
\begin{align} \label{e:FpFm}
(\Fp M)(x) &= \ker \Bigl(\bigoplus_{i=1}^{m} M(y_i) \to M(x) \Bigr) &&
(\Fp M)(y) = M(y)
\\ \notag
(\Fm M')(x) &= \coker \Bigl(M'(x) \to \bigoplus_{i=1}^{m} M'(y_i)
\Bigr) && (\Fm M')(y) = M'(y)
\end{align}
for $M \in \rep Q$, $M' \in \rep Q'$ and $y \in \Qx$, where the maps
$(\Fp M)(x) \to (\Fp M)(y_i)$ and $(\Fm M)(y_i) \to (\Fm M)(x)$ are
induced by the natural projection and inclusion.

It is clear that $\Fp$ is left exact and $\Fm$ is right exact. The
classical right derived functor of $\Fp$ takes the form
\begin{align} \label{e:R1Fp}
(R^1 \Fp M)(x) &= \coker \Bigl(\bigoplus_{i=1}^{m} M(y_i) \to M(x)
\Bigr) && (R^1 \Fp M)(y) = 0
\end{align}
hence $R^1 \Fp$ vanishes for modules not containing $S_x$ as direct
summand.

The total derived functors
\begin{align*}
R\Fp : \cD^b(Q) \to \cD^b(Q') && L\Fm : \cD^b(Q') \to \cD^b(Q)
\end{align*}
are triangulated equivalences, and their effect on the corresponding
cluster categories has been analyzed in~\cite{Zhu07}, where it is shown
that $R\Fp$ induces a triangulated equivalence $\cC_Q
\xrightarrow{\simeq} \cC_{Q'}$ whose action on the indecomposables of
$\cC_Q$ is given by
\begin{align} \label{e:CFpindec}
S_x \mapsto P'_x[1] && M \mapsto \Fp M && P_x[1] \mapsto S'_x && P_y[1]
\mapsto P'_y[1]
\end{align}
with an inverse given by
\begin{align} \label{e:CFmindec}
S'_x \mapsto P_x[1] && M' \mapsto \Fm M' && P'_x[1] \mapsto S_x &&
P'_y[1] \mapsto P_y[1]
\end{align}
for $M \neq S_x$, $M' \neq S'_x$ and $y \in \Qx$. Moreover, this
equivalence induces a bijection $\rho : \cT_{\cC_Q} \to \cT_{\cC_{Q'}}$
preserving the mutation graph~\cite[Proposition~3.2]{Zhu07}.

\begin{lemma} \label{l:rhorder}
Let $T, T' \in \cT_{\cC_Q}$. If $\rho(T) \leq \rho(T')$, then $T \leq
T'$.
\end{lemma}
\begin{proof}
By~\eqref{e:CFmindec}, $\fac T = \fac \Fm \wh{\rho(T)}$ if $P'_x[1]$ is
not a summand of $\rho(T)$, and $\fac T = \fac (S_x \oplus \Fm
\wh{\rho(T)})$ if $P'_x[1]$ is a summand of $\rho(T)$. Note that by
Lemma~\ref{l:cTxopen}, the latter case implies that $P'_x[1]$ is also a
summand of $\rho(T')$, hence in any case it is enough to verify that if
$M, N \in \rep Q'$ satisfy $\fac N \subseteq \fac M$, then $\fac F^{-}N
\subseteq \fac F^{-}M$. Indeed, since $F^{-}$ is right exact, it takes
an exact sequence $M^n \to N \to 0$ to an exact sequence $(F^{-} M)^n
\to F^{-} N \to 0$.
\end{proof}

\begin{prop} \label{p:rhocTx}
$\rho$ induces an isomorphism of posets $\cT^x_{\cC_Q}
\xrightarrow{\simeq} \cT^{x[1]}_{\cC_{Q'}}$.
\end{prop}
\begin{proof}
Note that by~\eqref{e:CFpindec}, $\rho(\cT^x_{\cC_Q}) =
\cT^{x[1]}_{\cC_{Q'}}$. In view of Lemma~\ref{l:rhorder}, it remains to
show that if $T, T' \in \cT^x_{\cC_Q}$ satisfy $T \leq T'$, then
$\rho(T) \leq \rho(T')$.

Write $\wh{T} = S_x \oplus U$ and $\wh{T'} = S_x \oplus U'$. Then $\fac
\rho(T) = \fac \Fp U$ and $\fac \rho(T') = \fac \Fp U'$, and we need to
show that $\Fp U' \in \fac \Fp U$.

Indeed, since $\fac T' \subseteq \fac T$, the proof of
Lemma~\ref{l:facfT} shows that $U' \in \fac U$, hence there exists a
short exact sequence
\[
0 \to K \to U^n \xrightarrow{\vphi} U' \to 0
\]
for some $n > 0$ and $K \in \rep Q$. Applying $\Hom_Q(-,S_x)$ to this
sequence, noting that $\Ext^1_Q(U', S_x) = 0$ since $T' \in
\cT^x_{\cC_Q}$, we get that $\Hom_Q(U^n, S_x) \to \Hom_Q(K, S_x)$ is
surjective, hence $K$ does not contain $S_x$ as summand (otherwise
$U^n$ would contain $S_x$ as summand). Therefore the exact sequence
\[
\Fp U^n \to \Fp U' \to R^1 \Fp K = 0
\]
shows that $\Fp U' \in \fac \Fp U$.
\end{proof}

\begin{prop} \label{p:rhocTxcomp}
$\rho$ induces an isomorphism of posets $\cT_{\cC_Q} \setminus
\cT^x_{\cC_Q} \xrightarrow{\simeq} \cT_{\cC_{Q'}} \setminus
\cT^{x[1]}_{\cC_{Q'}}$.
\end{prop}
\begin{proof}
For a representation $M \in \rep Q$, let $Q_M$ and $Q'_M$ be the
subquivers of $Q$ and $Q'$ obtained by deleting the vertices outside
$\supp M \cup \{x\}$. The quivers $Q'_M$ and $Q_M$ are related via a
BGP reflection at $x$, and we denote by $F^{+}_{Q_M} : \rep Q_M \to
\rep Q'_M$ the corresponding reflection functor. The restriction
functors $i^{-1} : \rep Q \to \rep Q_M$ and $j^{-1} : \rep Q' \to \rep
Q'_M$ induced by the natural embeddings $i : Q_M \to Q$ and $j : Q'_M
\to Q'$ satisfy
\[
j^{-1} \Fp M = F^{+}_{Q_M} i^{-1} M ,
\]
as can be easily verified using~\eqref{e:FpFm}.

As in the proof of Proposition~\ref{p:rhocTx}, it is enough to show
that if $T, T' \in \cT_{\cC_Q} \setminus \cT^x_{\cC_Q}$ satisfy $T \leq
T'$, then $\rho(T) \leq \rho(T')$. In view of the preceding paragraph,
we may assume that $Q = \supp \wh{T} \cup \{x\}$.

We consider two cases. First, assume that $x \in \supp \wh{T}$. Then $T
= \wh{T}$ is a tilting module, $\rho(T) = \Fp T$ and $\fac \rho(T') =
\fac \Fp \wh{T'}$ or $\fac \rho(T') = \fac (S'_x \oplus \Fp \wh{T'})$
according to whether $x \in \supp \wh{T'}$ or not, hence it is enough
to show that $S'_x \oplus \Fp \wh{T'} \in \fac \Fp T$.

By assumption, $\wh{T'} \in \fac T = T^{\perp}$. Since $T$ does not
contain $S_x$ as summand, $\Fp T$ is a tilting module and $\Fp \wh{T'}
\in (\Fp T)^{\perp} = \fac \Fp T$ \cite[Corollary~4.3]{Ladkani07t}.
Moreover, $S'_x \in \fac \Fp T$, as $\Fp T$ is sincere.

For the second case, assume that $x \not \in \supp \wh{T}$. Then $T =
P_x[1] \oplus \wh{T}$ and by Lemma~\ref{l:cTxopen}, $T' = P_x[1] \oplus
\wh{T'} \oplus P[1]$ where $P$ is a sum of projectives other than
$P_x$. By~\eqref{e:CFpindec}, $\rho(T) = S'_x \oplus \Fp \wh{T}$ and
$\rho(T') = S'_x \oplus \Fp \wh{T'} \oplus P'[1]$, hence it is enough
to show that $\Fp \wh{T'} \in \fac (S'_x \oplus \Fp \wh{T})$.

Indeed, since $\wh{T'} \in \fac \wh{T}$, there exists a short exact
sequence
\[
0 \to K \to \wh{T}^n \to \wh{T'} \to 0
\]
for some $n > 0$ and $K \in \rep Q$. Applying the functor $\Fp$, noting
that $\wh{T}$ does not contain $S_x$ as summand, we get
\[
0 \to \Fp K \to \Fp \wh{T}^n \to \Fp \wh{T'} \to R^1 \Fp K \to R^1 \Fp
\wh{T}^n = 0 .
\]
By~\eqref{e:R1Fp}, $R^1 \Fp K = {S'_x}^{n'}$ for some $n' \geq 0$,
hence $\Fp \wh{T'}$ is an extension of ${S'_x}^{n'}$ with a quotient of
$\Fp \wh{T}^n$. The result now follows, as $\fac (S'_x \oplus \Fp
\wh{T})$ is closed under extensions.
\end{proof}

\begin{cor} \label{c:clcomm}
We have a commutative diagram
\[
\xymatrix{
\cT^x_{\cC_Q} \ar[r]^{\rho}_{\simeq} \ar[d]^{f} &
\cT^{x[1]}_{\cC_{Q'}} \ar[d]^{g} \\
\cT_{\cC_Q} \setminus \cT^x_{\cC_Q} \ar[r]^{\rho}_{\simeq} &
\cT_{\cC_{Q'}} \setminus \cT^{x[1]}_{\cC_{Q'}}
}
\]
\end{cor}
\begin{proof}
By Propositions~\ref{p:rhocTx} and~\ref{p:rhocTxcomp}, $\rho$ induces
the two horizontal isomorphisms. For $T \in \cT^x_{\cC_Q}$, $f(T)$ is
defined as the mutation of $T$ with respect to $S_x$ and $g(\rho(T))$
is defined as the mutation of $\rho(T)$ with respect to $P'_x[1]$,
which is, by~\eqref{e:CFpindec}, the image of $S_x$ under the
triangulated equivalence $\cC_Q \to \cC_{Q'}$. Therefore the
commutativity of the diagram follows by the fact that $\rho$ preserves
the mutation graph~\cite[Proposition~3.2]{Zhu07}.
\end{proof}

\begin{theorem}
The posets $\cT_{\cC_Q}$ and $\cT_{\cC_{Q'}}$ are related via a
flip-flop.
\end{theorem}
\begin{proof}
Use Corollaries~\ref{c:ffclsink}, \ref{c:ffclsource}
and~\ref{c:clcomm}.
\end{proof}

\section{Example}
\label{sec:clexample}

Consider the following two quivers $Q$ and $Q'$ whose underlying graph
is the Dynkin diagram $A_3$. The quiver $Q'$ is obtained from $Q$ by a
BGP reflection at the sink $3$.
\begin{align*}
Q : \xymatrix{
{\bullet_1} \ar[r] & {\bullet_2} \ar[r] & {\bullet_3}
} &&
Q' : \xymatrix{
{\bullet_1} \ar[r] & {\bullet_2} & {\bullet_3} \ar[l]
}
\end{align*}

We denote the indecomposables of the cluster categories $\cC_Q$ and
$\cC_{Q'}$ by specifying their dimension vectors. These consist of the
positive roots of $A_3$, which correspond to the indecomposable
representations of the quivers, together with the negative simple roots
$-e_1, -e_2, -e_3$ which correspond to the shifted projectives.

\begin{figure}
\[
\xymatrix@=1.5pc{
& &
{\begin{smallmatrix} 
0 & 1 & 0 \\ 0 & 1 & 1 \\ 1 & 1 & 1
\end{smallmatrix}} \ar[d] \ar[dddrrr] \\
& & {\begin{smallmatrix} 
0 & 1 & 0 \\ 1 & 1 & 0 \\ 1 & 1 & 1
\end{smallmatrix}} \ar[dl] \ar[dr] \\
{\begin{smallmatrix} 
\mathbf{0} & \mathbf{0} & \mathbf{1} \\
\mathbf{0} & \mathbf{1} & \mathbf{1} \\
\mathbf{1} & \mathbf{1} & \mathbf{1}
\end{smallmatrix}} \ar[uurr] \pmb{\ar[d]} \pmb{\ar[dddrrr]} &
{\begin{smallmatrix} 
1 & 0 & 0 \\ 1 & 1 & 0 \\ 1 & 1 & 1
\end{smallmatrix}} \ar[dr] & &
{\begin{smallmatrix} 
0 & 1 & 0 \\ 1 & 1 & 0 \\ 0 & 0 & -1
\end{smallmatrix}} \ar[dl] \ar[ddrr] \\
{\begin{smallmatrix} 
\mathbf{0} & \mathbf{0} & \mathbf{1} \\
\mathbf{1} & \mathbf{0} & \mathbf{0} \\
\mathbf{1} & \mathbf{1} & \mathbf{1}
\end{smallmatrix}} \ar[ur] \pmb{\ar[ddrr]} & &
{\begin{smallmatrix} 
1 & 0 & 0 \\ 1 & 1 & 0 \\ 0 & 0 & -1
\end{smallmatrix}} \ar[dr] & & &
{\begin{smallmatrix} 
0 & 1 & 0 \\ 0 & 1 & 1 \\ -1 & 0 & 0
\end{smallmatrix}} \ar[d] \\
& & &
{\begin{smallmatrix} 
1 & 0 & 0 \\ 0 & -1 & 0 \\ 0 & 0 & -1
\end{smallmatrix}} \ar[dr] & &
{\begin{smallmatrix} 
0 & 1 & 0 \\ 0 & 0 & -1 \\ -1 & 0 & 0
\end{smallmatrix}} \ar[dl] \\
& &
{\begin{smallmatrix} 
\mathbf{0} & \mathbf{0} & \mathbf{1} \\
\mathbf{1} & \mathbf{0} & \mathbf{0} \\
\mathbf{0} & \mathbf{-1} & \mathbf{0}
\end{smallmatrix}} \ar[ur] \pmb{\ar[dr]} &
{\begin{smallmatrix} 
\mathbf{0} & \mathbf{0} & \mathbf{1} \\
\mathbf{0} & \mathbf{1} & \mathbf{1} \\
\mathbf{-1} & \mathbf{0} & \mathbf{0}
\end{smallmatrix}} \ar[uurr] \pmb{\ar[d]} &
{\begin{smallmatrix} 
-1 & 0 & 0 \\ 0 & -1 & 0 \\ 0 & 0 & -1
\end{smallmatrix}}  \\
& & &
{\begin{smallmatrix} 
\mathbf{0} & \mathbf{0} & \mathbf{1} \\
\mathbf{0} & \mathbf{-1} & \mathbf{0} \\
\mathbf{-1} & \mathbf{0} & \mathbf{0}
\end{smallmatrix}} \ar[ur]
}
\]

\[
\xymatrix@=1.5pc{
& &
{\begin{smallmatrix} 
0 & 1 & 1 \\ 0 & 1 & 0 \\ 1 & 1 & 0
\end{smallmatrix}} \ar[ddll] \ar[d] \ar[dddrrr] \\
& &
{\begin{smallmatrix} 
0 & 1 & 1 \\ 1 & 1 & 1 \\ 1 & 1 & 0
\end{smallmatrix}} \ar[dl] \ar[dr] \\
{\begin{smallmatrix} 
\mathbf{0} & \mathbf{0} & \mathbf{-1} \\
\mathbf{0} & \mathbf{1} & \mathbf{0} \\
\mathbf{1} & \mathbf{1} & \mathbf{0}
\end{smallmatrix}} \pmb{\ar[d]} \pmb{\ar[dddrrr]} &
{\begin{smallmatrix} 
1 & 0 & 0 \\ 1 & 1 & 1 \\ 1 & 1 & 0
\end{smallmatrix}} \ar[dl] \ar[dr]
& &
{\begin{smallmatrix} 
0 & 1 & 1 \\ 1 & 1 & 1 \\ 0 & 0 & 1
\end{smallmatrix}} \ar[dl] \ar[ddrr] \\
{\begin{smallmatrix} 
\mathbf{0} & \mathbf{0} & \mathbf{-1} \\
\mathbf{1} & \mathbf{0} & \mathbf{0} \\
\mathbf{1} & \mathbf{1} & \mathbf{0}
\end{smallmatrix}} \pmb{\ar[ddrr]} & &
{\begin{smallmatrix} 
1 & 0 & 0 \\ 1 & 1 & 1 \\ 0 & 0 & 1
\end{smallmatrix}} \ar[dr] & & &
{\begin{smallmatrix} 
0 & 1 & 1 \\ 0 & 1 & 0 \\ -1 & 0 & 0
\end{smallmatrix}} \ar[ddll] \ar[d] \\
& & &
{\begin{smallmatrix} 
1 & 0 & 0 \\ 0 & -1 & 0 \\ 0 & 0 & 1
\end{smallmatrix}} \ar[dl] \ar[dr] & &
{\begin{smallmatrix} 
0 & 1 & 1 \\ 0 & 0 & 1 \\ -1 & 0 & 0
\end{smallmatrix}} \ar[dl] \\
& &
{\begin{smallmatrix} 
\mathbf{0} & \mathbf{0} & \mathbf{-1} \\
\mathbf{1} & \mathbf{0} & \mathbf{0} \\
\mathbf{0} & \mathbf{-1} & \mathbf{0}
\end{smallmatrix}} \pmb{\ar[dr]} &
{\begin{smallmatrix} 
\mathbf{0} & \mathbf{0} & \mathbf{-1} \\
\mathbf{0} & \mathbf{1} & \mathbf{0} \\
\mathbf{-1} & \mathbf{0} & \mathbf{0}
\end{smallmatrix}} \pmb{\ar[d]} &
{\begin{smallmatrix} 
-1 & 0 & 0 \\ 0 & -1 & 0 \\ 0 & 0 & 1
\end{smallmatrix}} \ar[dl] \\
& & &
{\begin{smallmatrix} 
\mathbf{0} & \mathbf{0} & \mathbf{-1} \\
\mathbf{0} & \mathbf{-1} & \mathbf{0} \\
\mathbf{-1} & \mathbf{0} & \mathbf{0}
\end{smallmatrix}}
}
\]
\caption{Hasse diagrams of the posets $\cT_{\cC_Q}$ (top) and
$\cT_{\cC_{Q'}}$ (bottom).} \label{fig:cltilting}
\end{figure}

Figure~\ref{fig:cltilting} shows the Hasse diagrams of the posets
$\cT_{\cC_Q}$ and $\cT_{\cC_{Q'}}$, where we used bold font to indicate
the subsets $\cT^3_{\cC_Q}$ and $\cT^{3[1]}_{\cC_{Q'}}$ of cluster
tilting objects containing the simple $S_3$ and the shifted projective
$P'_3[1]$ as summand, respectively.

The posets $\cT_{\cC_Q}$ and $\cT_{\cC_{Q'}}$ are Cambrian lattices,
and can be realized as sublattices of the weak order on the group of
permutations on $4$ letters, see~\cite[Section~6]{Reading06}. Moreover,
the underlying graph of their Hasse diagrams is the $1$-skeleton of the
three-dimensional Stasheff associhedron. $\cT_{\cC_Q}$ is a Tamari
lattice, corresponding to the linear orientation on $A_3$.

The BGP reflection at the vertex $3$, whose action on the dimension
vectors is given by
\[
v \mapsto \begin{cases} v & \text{if $v \in \{-e_1, -e_2\}$} \\
s_3(v) & \text{otherwise}
\end{cases}
\]
where $s_3$ is the linear transformation specified by
\[
s_3(v) = v \cdot
\begin{pmatrix}
1 & 0 & 0 \\
0 & 1 & 1 \\
0 & 0 & -1 \\
\end{pmatrix} ,
\]
induces isomorphisms $\cT^3_{\cC_Q} \xrightarrow{\simeq}
\cT^{3[1]}_{\cC_{Q'}}$ and $\cT_{\cC_Q} \setminus \cT^3_{\cC_Q}
\xrightarrow{\simeq} \cT_{\cC_{Q'}} \setminus \cT^{3[1]}_{\cC_{Q'}}$
compatible with the mutations at $S_3$ and $P'_3[1]$.

\clearpage



\providecommand{\bysame}{\leavevmode\hbox to3em{\hrulefill}\thinspace}
\providecommand{\MR}{\relax\ifhmode\unskip\space\fi MR }
\providecommand{\MRhref}[2]{%
  \href{http://www.ams.org/mathscinet-getitem?mr=#1}{#2}
} \providecommand{\href}[2]{#2}

\end{document}